\documentclass[12pt,a4paper]{amsart}

\calclayout

\usepackage{xcolor}
\usepackage[lite]{amsrefs}
\usepackage{amssymb}
\usepackage[all,cmtip]{xy}

\usepackage{mathrsfs}
\usepackage{tabularx}
\usepackage{booktabs}
\usepackage[labelfont=bf,format=plain,justification=raggedright,singlelinecheck=false]{caption}
\usepackage{bookmark}
\usepackage{hyperref}
\usepackage{comment}
\usepackage{enumitem}

\newcommand{\F}{\mathbb{F}}

\newcommand{\C}{\mathcal{C}}

\newcommand{\Gal}{\mathrm{Gal}}

\newcommand{\PGL}{\mathrm{PGL}}
\newcommand{\PSL}{\mathrm{PSL}}
\newcommand{\aut}{\mathrm{Aut}}
\newcommand{\Fq}{\mathbb{F}_q}
\newcommand{\Fqq}{\mathbb{F}_{q^2}}

\newcommand{\mZ}{\mathbb{Z}}
\newcommand{\mN}{\mathbb{N}}

\newcommand{\tZ}{\widetilde{Z}}

\numberwithin{equation}{section}

\theoremstyle{plain}
\newtheorem{theorem}[equation]{Theorem}

\newtheorem{lemma}[equation]{Lemma}

\theoremstyle{definition}

\theoremstyle{remark}
\newtheorem{remark}[equation]{Remark}
\newtheorem{example}[equation]{Example}

\begin{document}

\title[Double A.-S. extensions with many lifted automorphisms]{Double Artin-Schreier extensions of rational function fields with many lifted automorphisms}

\author{Herivelto Borges}
\author{Jonathan Niemann}
\author{Giovanni Zini}

\begin{abstract}
In this paper we investigate algebraic function fields in positive characteristic mainly obtained as double Artin-Schreier extensions of rational function fields with a plane model. The goal is to extend to such extensions large automorphism groups of the rational function field. In this way, we construct some new families of ordinary function fields and determine their full automorphism groups. Such groups are large with respect to the genus, compared with the known upper bounds on the size of the automorphism group of an ordinary function field.
\end{abstract}

\maketitle

\section{Introduction}

An algebraic function field (of transcendence degree one) over an algebraically closed field $K$ is a field extension $F/K$ such that $F/K(x)$ is finite for some $x \in F$, where $x$ is transcendental over $K$. Given an irreducible 
algebraic curve over $K$, the rational functions on the curve form such a function field, and every algebraic function field arises in this way. We will mostly use the language of function fields, but our results could also be stated in terms of curves.

Automorphisms of algebraic function fields $F/K$, that is, field automorphisms of $F$ that fix the base field $K$ elementwise, have been studied extensively in the past decades. In particular, there has been a lot of work on relating the size of the automorphism group $\aut(F)$ with other invariants of the function field, such as the genus $g$ of $F$ and, when $F$ has positive characteristic $p$, the $p$-rank (or Hasse-Witt invariant) $\gamma$ of $F$.
If $F$ is rational or elliptic, then the automorphism group is infinite, so suppose from now on that $g\geq 2$. In this case, it is known that $\aut(F)$ is finite, and when $K = \mathbb{C}$ the classical Hurwitz bound (1893) states that
$$
    |\aut(F)| \leq 84(g-1).
$$
The bound fails in positive characteristic $p$; in this case, either 
$$
    |\aut(F)| < 8g^3,
$$
or $F$ is a function field from one of four specific families (see \cite{Stichtenoth_1973} and \cite{Henn_1978}). These four special types of function fields all have $p$-rank equal to zero. In fact, if the characteristic is odd and the genus is even then it is known (see \cites{Giulietti_Korchmaros_2019,Montanucci_2022}) that 
$$
    |\aut(F)| \geq 900g^2 \Longrightarrow \gamma = 0.
$$
As the $p$-rank always satisfies $0\leq\gamma\leq g$, it is natural to go to the other extreme: how large can the automorphism group be if the $p$-rank is as large as possible, i.e., if $\gamma=g$? In this case $F$ is called \emph{ordinary}, and it was shown by Nakajima in \cite{nakajima_p-ranks_1987} that 
$$
    |\aut(F)| \leq 84g(g-1)
$$
when $F$ is ordinary and $g\geq 2$. The bound has been refined, see \cite{Gunby_Smith_Yuan_2015} and \cite{Lia_Timpanella_2021}, however the order of magnitude is still $g^2$. 

When $p$ is odd and $g$ is even, there is an asymptotically sharper bound involving $g^{7/4}$: under this assumptions, it was shown in \cite{Montanucci_Speziali_odd_char_2020} that ordinary function fields satisfy
$$
    |\aut(F)| < 821.37 \cdot g^{7/4}.
$$
Under a different assumption, namely that the automorphism group of the ordinary function field is solvable, it is known that
$$
    |\aut(F)| \leq 35(g+1)^{3/2}
$$
in odd characteristic (see \cite{Korchmaros_Montanucci_2019}), and also in even characteristic when $g$ is even (see \cite{Montanucci_Sepziali_2019}). Moreover, there are examples of function fields attaining this bound up to the constant, see \cites{Cornelissen_2001,MZ_gen_Artin-Mumford, Korchmaros_Montanucci_Speziali_2018,Borges_Korchmaros_Speziali_2024}. 

To the best of our knowledge, there are no upper bounds on the size of $\aut(F)$ which are valid for every ordinary function field $F$ and sharp. However there are examples where the order of magnitude of $\aut(F)$ exceeds $g^{3/2}$. One such example is the Dickson-Guralnick-Zieve function field over the prime field $\mathbb{F}_p$, where the size of the automorphism group is roughly $g^{8/5}$, see \cite{Giulietti_Korchmaros_Timpanella_2019}. Another example, which we will return to later, is a function field in characteristic $2$ due to Zieve \cite{Zieve_private}, whose automorphism group has size roughly $g^{5/3}$. On the path to finding sharp upper bounds, it remains interesting in general to construct examples of function fields where the order of magnitude of the automorphism group is $g^{1 + \varepsilon}$, for some $\varepsilon > 0$.

In this paper, we study a certain type of algebraic function fields, namely those that can be obtained as a double, or in one case triple, Artin-Schreier extension over a rational function field. We will see several examples of both known and new ordinary function fields with large automorphisms groups that can be obtained in this way. We recover certain function fields, most notably what we call the Singer function field (which is related to a curve first introduced by Nakajima in \cite{nakajima_p-ranks_1987}), and the Zieve function field mentioned above; beside that, our main contribution is the construction and study of the so-called modified and extended Zieve function fields, as well as the determination of the full automorphism group for the Singer, Zieve and extended Zieve function fields.

The paper is organized as follows. In Section \ref{sec_preliminary} we collect some general results on double Artin-Schreier extensions. Section \ref{sec:conics} contains a study of the case where the right-hand sides of the two defining equations are coordinate functions of an irreducible conic. In particular, we obtain the Singer function field in this way. In Section \ref{sec:three_Zieve} we study three ordinary function fields related to a curve introduced by Zieve. Finally, we determine the full automorphism group of the Singer, Zieve and extended Zieve function fields in Section \ref{sec_aut_full}.

\section{Background and preliminary results}\label{sec_preliminary}


Let $q = p^n$ for some prime $p$ and $n \in \mZ_{\geq 0}$, and let $K = \overline{\mathbb{F}}_p$ be an algebraic closure of $\mathbb{F}_p$.
We investigate function fields of the type $F := K(y,z)$ where
\begin{align*}
    y^q - y &= u, \\
    z^q - z &= v, 
\end{align*}
and $K(u,v)$ is a rational function field defined by an equation 
$$
    f(u,v) = 0,
$$
where $f \in K[u,v]$ is irreducible.

Equivalently, the affine equation $f(U,V) = 0$ defines an irreducible plane curve, and $F$ is the function field of the irreducible curve $\C$ given by
\begin{align*}
\C: \begin{cases}
    Y^q-Y &= U, \\
    Z^q-Z &= V, \\
    f(U,V) &= 0. 
\end{cases}
\end{align*}

The automorphism group of such a function field, or curve, always contains an elementary abelian subgroup of order $q^2$, namely the $\Fq$-translations in $y$ and $z$. Moreover, in certain cases some automorphisms of the rational function field $K(u,v)$ can be ``lifted'' to $F=K(y,z)$, that is, extended to automorphisms of $F$. 

We consider only the case where $f(u,v)$ is chosen in such a way that $\lbrack F : K(u,v) \rbrack = q^2$ (see Lemma \ref{lem_one_eq_irr}). By the theory of Artin-Schreier extensions, the ramification in $F/K(u,v)$ depends on the poles of $u$ and $v$. If there is just one pole then the ramification is easy to determine, as the following lemma shows.

\begin{lemma}\label{lem_p_rank_0_one_pole}
    Suppose $u$ and $v$ have the same unique pole in $K(u,v)$, and denote this pole by $P_\infty$. If $\lbrack F : K(u,v) \rbrack = q^2$, then $P_\infty$ is totally ramified in $F/K(u,v)$, it is the only ramified place in $F/K(u,v)$, and the $p$-rank of $F$ is $0$.
\end{lemma}

\begin{proof}
    By Artin-Schreier theory (see e.g. \cite{GS_p_extensions_1991}), ramification in $F/K(u,v)$ can occur only above the poles of $u$ and $v$, so only above $P_\infty$.
    Recall that $\Gal(F/K(u,v))$ acts transitively on the places of $F$ lying above $P_\infty$, and let $\ell\geq1$ be the number of places of $F$ lying above $P_\infty$. Then the Deuring-Shafarevich formula implies 
    $$
        \gamma(F) - 1 =q^2\left(\gamma(K(u,v))-1\right) + (q^2 - \ell) = - \ell.
    $$
    As $\gamma(F)\geq0$, this implies $\ell=1$ and $\gamma(F) = 0$.
\end{proof}

We are mostly interested in ordinary function fields, so we should pick $f(u,v)$ in such a way that $u$ and $v$ together have more than one pole. However, in the end of Section \ref{sec:conics}, we will briefly mention two remarkable examples with $p$-rank zero. To determine the $p$-rank and genus in other cases, it will be useful to have a single defining equation for $F$ over $K(u,v)$. In the following lemma we find such an equation when $[F:K(u,v)]=q^2$.

\begin{lemma}\label{lem_one_eq_irr}
    Let $\eta \in \Fqq \setminus \Fq$ and $t = y + \eta z\in F$. Then $F = K(u,v,t)$ and
    \begin{equation}\label{eq:single}
        t^{q^2} - t = u^q - u + \eta (v^q - v).
    \end{equation}
    Also, $[F:K(u,v)]=q^2$ if and only if the polynomial
    $$
        T^p - T - \mu \left( u^q - u + \eta(v^q - v) \right) \in K(u,v)[T],
    $$
    is irreducible for every $\mu \in \Fqq$. 
\end{lemma}

\begin{proof}
    Equation \eqref{eq:single} is easy to check. Now, it is clear that $K(u,v,t) \subseteq F$. In order to show $F\subseteq K(u,v,t)$, note first that 
    $t^q - t = y^q - y + \eta^q z^q - \eta z = u +  \eta^q z^q - \eta z$, 
    so $\eta^q z^q - \eta z \in K(u,v,t)$. Together with $z^q - z = v \in K(u,v,t)$ and $\eta^q\ne\eta$, this implies $z \in K(u,v,t)$. Similarly, $(\eta^{-1}t)^q - \eta^{-1}t = (\eta^{-1})^q y^q - \eta^{-1}y + v \in K(u,v,t)$
    and $y^q - y = u \in K(u,v,t)$ imply $y \in K(u,v,t)$. The rest of the claim is \cite{GS_p_extensions_1991}*{Lemma 1.3}. 
\end{proof}

The method described in \cite{GS_p_extensions_1991} can now be used to determine the genus and $p$-rank of $F$. There are $m = (q^2-1)/(p-1)$ distinct subfields of $F$ of degree $p$ over $K(u,v)$. Namely, if $\{\mu_1, \dots, \mu_m\} \subseteq \Fqq^\times$ is a set of representatives for the cosets of $\Fqq^*$ modulo $\F_p^*$, then such intermediate fields are of the form $F_{\mu_i} := K(u,v,t_{\mu_i})$ with
$$
    t_{\mu_i}^p - t_{\mu_i} = \mu_i \left( u^q - u + \eta(v^q - v) \right)
$$
for $i = 1, \dots, m$. By \cite{GS_p_extensions_1991}*{Theorem 2.1}, the genus of $F$ is given by 
\begin{equation}\label{eq_genus_sum}
    g(F) = \sum_{i=1}^m g(F_{\mu_i}).
\end{equation}
The idempotent relation given in the proof of \cite{GS_p_extensions_1991}*{Theorem 2.1} also implies that
\begin{equation}\label{eq_p-rank_sum}
    \gamma(F) = \sum_{i=1}^m \gamma(\mu_i),
\end{equation}
by the Kani-Rosen theorem (see \cite{Kani_Rosen_Jacobians}).

The genus and $p$-rank of each $F_{\mu_i}$ can be determined using \cite{Sti}*{Proposition 3.7.8} and the Deuring-Shafarevich formula (see \cite{HKT}*{Theorem 11.62}). The first step in this direction is to define
\begin{align*}
    \tilde{t}_{\mu_i} := t_{\mu_i} &- \left(\mu_i^{1/p}u^{q/p} + \mu_i^{1/p^2}u^{q/p^2} + \dots + \mu_i^{1/q} u\right) \\
            &- \left((\eta \mu_i)^{1/p}v^{q/p} + (\eta\mu_i)^{1/p^2}v^{q/p^2} + \dots + (\eta\mu_i)^{1/q} v\right), 
\end{align*}

where the $p^k$-roots are taken in $K$. Then, $F_{\mu_i} = K(u,v,\tilde{t}_{\mu_i})$ and a direct check shows 
\begin{equation}\label{eq_F_mu_i}
    \tilde{t}_{\mu_i}^p - \tilde{t}_{\mu_i} = \left(\mu_i^q -  \mu_i\right) u + \left( (\mu_i\eta)^q - \mu_i\eta \right) v.
\end{equation}

In the coming sections, we will see examples where both the genus and $p$-rank of each $F_{\mu_i}$ can be determined directly from this expression once the poles of $u$ and $v$ are known.

\section{Double Artin-Schreier extensions over conics}\label{sec:conics}

In this section, we consider the case where $f(u,v)$ is irreducible and of total degree $2$, i.e., where the affine equation $f(u,v) = 0$ gives rise to an irreducible conic. Several well-known and interesting function fields arise in this way, in particular we recover two examples of ordinary function fields with many automorphisms. If the conic has just one point at infinity then we already know from Lemma \ref{lem_p_rank_0_one_pole} that the $p$-rank of $F$ is zero. This case is considered in Example \ref{ex_1_rational} and \ref{ex_1_non_rational}. If instead there are two points at infinity then it turns out that $F$ is ordinary.


\begin{lemma}\label{lem_two_points_ordinary}
    Suppose $f(u,v) = 0$ gives rise to an irreducible conic with two points at infinity. Then, the genus and $p$-rank of $F$ are equal, i.e., $F$ is ordinary. 
\end{lemma}

\begin{proof}
    Let $P_1$ and $P_2$ be the two places of $K(u,v)$ that correspond to the points at infinity. These are the only possible poles of $u$ and $v$ and they are simple poles. Hence, the right hand side of Equation \ref{eq_F_mu_i} has valuation at least $-1$ at all places. By \cite[Proposition 3.7.8]{Sti} and Hilbert's Different Formula this means that there is no higher ramification in $F_{\mu_i}/K(u,v)$, i.e., the second ramification group is always trivial. In particular, the Deuring-Shafarevich formula and the Hurwitz genus formula coincide, so $\gamma(\F_{\mu_i}) = g(F_{\mu_i})$. The result now follows from Equations \ref{eq_genus_sum} and \ref{eq_p-rank_sum}.
\end{proof}

As a first example of this type, we recover the function field of the famous Artin-Mumford curve. 

\begin{example}[Artin-Mumford]\label{ex_am}
    Suppose $f(u,v) = uv - 1$. The conic given by this equation has two distinct $\Fq$-rational points at infinity corresponding to the pole of $u$ and $v$ respectively. By combining the two equations we see that $K(y,z)$ can be defined by the sole equation $(y^q-y)(z^q-z) = 1$, so it is the well known Artin-Mumford function field (see \cite{MZ_gen_Artin-Mumford}). The genus is $(q-1)^2$, it is ordinary, and the automorphism group is the semi-direct product of an elementary abelian group of order $q^2$ and a dihedral group of order $2(q-1)$. The elementary abelian group is given by translations in $y$ and $z$ and the dihedral group can be obtained by lifting the automorphisms of $K(u,v)$ given by $(u,v) \mapsto (au, a^{-1}v)$, for $a \in \Fq^\times$ and the involution $(u,v) \mapsto (v,u)$. The order of magnitude of the automorphism group in terms of the genus is roughly $g^{3/2}$.
\end{example}

If one of the two points at infinity is non-rational then we obtain an ordinary function field of genus $q^2 -q$:

\begin{example}[One rational and one non-rational point at infinity]
    Suppose $f(u,v) = u^2 + \varepsilon uv + 1$, for some $\varepsilon \in K \setminus \Fq$. The conic given by this equation has two points at infinity; one is $\Fq$-rational and the other is non-rational. Denote the corresponding places by $P_1$ and $P_2$, with $P_1$ being the place that corresponds to the $\Fq$-rational point. Then, 
    $$
        v_{P_1}(u) = 1\ \text{ and } \ v_{P_2}(u) = -1,
    $$
    while
    $$
        v_{P_1}(v) = -1\ \text{ and } \ v_{P_2}(v) = -1.
    $$

    Moreover, we have $v_{P_2}(u+\varepsilon v) = 1$ so the right hand side of Equation \ref{eq_F_mu_i} has valuation $-1$ at $P_2$ unless
    $$
        \frac{\mu_i^{q} - \mu_i}{(\mu_i\eta)^{q} - \mu_i\eta} =\frac{1}{\varepsilon},
    $$
    but the left hand side is in $\Fq$, so this does not happen. The valuation at $P_1$ is $-1$ except if $\mu_i\eta \in \Fq$, so we get
    $$
        g(F) = \left(\frac{q^2-1}{p-1} - \frac{q-1}{p-1}\right) \cdot  (p-1) = q^2 - q. 
    $$
    The single defining equation one obtains using Lemma \ref{lem_one_eq_irr} is not particularly nice in this case, and it is not clear to us which automorphisms can be lifted, if any.
\end{example}

Finally, when both points at infinity are non-rational we obtain an ordinary function field of genus $q^2-1$ for which a Singer subgroup of $\aut(K(u,v))$ can be lifted. 

\begin{example}[Singer]\label{ex_singer}
        Suppose $p \neq 2$ and consider the conic given by $f(u,v) = u^2 - \varepsilon  v^2 - 1$, where $\varepsilon$ is a non-square element of $\Fq$. This conic has two non-$\Fq$-rational points at infinity, corresponding to two places, $P_1$ and $P_2$. The valuation of $u$ and $v$ is $-1$ at both $P_1$ and $P_2$, and we may assume that $u+\sqrt \varepsilon v$ and $u - \sqrt\varepsilon v$ have zeros at $P_1$ and $P_2$ respectively, where $\sqrt{\varepsilon}$ is some fixed square root of $\varepsilon$ in $\Fqq$. Since $\sqrt\varepsilon \notin \Fq$, we get that the right hand side of Equation $\ref{eq_F_mu_i}$ has valuation $-1$ at both $P_1$ and $P_2$. This means that 
        $$
            g(F) = \frac{q^2-1}{p-1} \cdot (p-1) = q^2 - 1.
        $$
        To obtain a nice single defining equation for this function field, we define $s := y-\sqrt\varepsilon \cdot z$ and $t := y +\sqrt\varepsilon \cdot z$. It is clear that $F = K(y,z) = K(s,t)$, and a straight forward calculation, using $(\sqrt\varepsilon)^q + \sqrt\varepsilon = 0$, shows that
        $$
            (s^q-t)(t^q - s) = 1. 
        $$

        From this equation, we immediately see that the automorphism group of $F$ contains a cyclic group of order $q+1$ consisting of the automorphisms $\delta_\lambda:(s,t) \mapsto (\lambda s, \lambda^{-1} t)$, for $\lambda \in \Fqq$ with $\lambda^{q+1}=1$. These automorphisms can be obtained by lifting a Singer subgroup from the automorphisms of $K(u,v)$. To see this, define $x := t^q - s$ and note that 
        $$
            u = \frac{1}{2}\left(x + \frac{1}{x}\right) \ \text{ and } \ v =  \frac{-1}{2\sqrt{\varepsilon}}\left(x - \frac{1}{x}\right).
        $$

        This means that $K(u,v) = K(x)$, and we can see directly that the $\gamma_\lambda$'s act as a Singer group on $K(x)$ since 
        $$
            \gamma_\lambda(x) = \lambda x.
        $$

        Together with the involution $\pi:(s,t) \mapsto (t,s)$, which interchanges $x$ and $1/x$, this gives a dihedral group of order $2(q+1)$. We will see in Section \ref{sec_aut_full}, for $q > 7$, that the full automorphism group of $F$ is the semidirect product of the usual elementary abelian group of order $q^2$, i.e., the $\Fq$-translations in $y$ and $z$, and this dihedral group. For now, we just note that the automorphism group is always of order at least $2q^2(q+1)$ which is approximately $2g(F)^{3/2}$.
\end{example}

\begin{remark}\label{remark_even_singer}
    For $p=2$, one can obtain something similar by considering $f(u,v) = u^2 + \varepsilon v^2 + uv + 1$, where $\varepsilon \in \Fq$ is chosen such that $\mathrm{Tr}_{\Fq/\F_2}(\varepsilon) = 1$. The corresponding conic has two non-rational points at infinity, $F$ is ordinary of genus $q^2-1$ and we can still lift a Singer subgroup. In fact, one can repeat the above argument with $\sqrt{\varepsilon}$ replaced by a fixed root, $\xi$, of the polynomial $X^2 + X + \varepsilon \in \F_q[X]$. Defining $s := y + \xi \cdot z$ and $t := y + (\xi + 1) \cdot z$, we get $F = K(s,t)$ and it can be checked directly that $(s^q + t)(t^q + s) = 1$.
\end{remark}

\begin{remark}
    One can check, for example by using the equation in Lemma \ref{lem_one_eq_irr}, that the function field studied above corresponds to the curve which appears in the remark after Theorem 3 in Nakajima's paper (\cite{nakajima_p-ranks_1987}).
\end{remark}

The large automorphism groups in Example \ref{ex_am} and Example \ref{ex_singer} are of the form $E \rtimes H$, where $E$ is an elementary abelian subgroup of order $q^2$ and $H$ is isomorphic to a subgroup of $\PGL(2,q)$. The automorphisms in $E$ are the $\Fq$-translations in $y$ and $z$, and $H$ consists of automorphisms ``lifted'' from $\aut_{\Fq}(K(u,v))$. Since $E$ is the Galois group of the extension $F/K(u,v)$, the group $H$ is isomorphic to a subgroup of $\aut(K(u,v))$, and it acts on the ramified places of $K(u,v)$. 

When $f(u,v)=0$ is an irreducible conic, this means that the action of $H$ on the places of $K(u,v)$ has at least one orbit of order at most two. By \cite[Theorem 11.91]{HKT}, this limits the options for the type of group $H$ can be. In particular, $H$ cannot be isomorphic to $\PGL(2,q)$. In the next section, we consider examples where the set of poles of $u$ and $v$ is of size $q+1$ instead, to allow for the possibility of $H \simeq PGL(2,q)$. \newline

However, we first complete the discussion of double Artin-Schreier extensions over concics by dealing with the cases where the conic has just one point at infinity.

\begin{example}[One $\Fq$-rational point at infinity]\label{ex_1_rational}
    Suppose $f(u,v) = u^2 - \varepsilon v$, for some $\varepsilon \in K^\times$. The conic given by this equation has just one point at infinity, it is $\Fq$-rational, and it corresponds to the unique pole of $u$ and $v$ in $K(u,v)$. We denote this pole by $P_\infty$ and note that
    $$
        v_{P_\infty} (u) = -1 \ \text{ and } \ v_{P_\infty}(v) = -2.
    $$
    We know from Lemma \ref{lem_p_rank_0_one_pole} that the $p$-rank of $F$ is zero. For $p=2$ it follows from Equation \ref{eq_F_mu_i} and the preceding discussion that $g(F) = 0$. For $p$ odd we get instead
    $$
        g(F_{\mu_i}) = 
        \begin{cases}
            \hfil 0  &\text{if }  \mu_i\eta \in \Fq, \text{ and} \\
            \frac{p-1}{2} &\text{else.} 
        \end{cases}
    $$
    Exactly $\frac{q-1}{p-1}$ of the $\mu_i$ satisfy $\mu_i \eta \in \F_q$, so 
    $$
        g(F) = \sum_{i= 1}^m g(F_{\mu_i}) = \left(\frac{q^2-1}{p-1}-\frac{q-1}{p-1}\right) \frac{p-1}{2} = \frac{q^2-q}{2}.
    $$

    If $\varepsilon \in \Fq$, and $p$ is still assumed to be odd, then $\tilde{z} := \frac{1}{2}(y^2 - \varepsilon z)$ satisfies
    $$
        \tilde{z}^q - \tilde{z} = y(y^q - y), 
    $$
    and $F = K(y,\tilde{z})$. This is a well-known function field which has several special properties; it has, for example, an automorphism group of order $q^3(q-1)$ and for $p\equiv 3 \pmod 4$ it is $\F_{q^{2p}}$-maximal.
    There are automorphisms of $K(u,v)= K(u)$ that can be lifted, namely those given by $u \mapsto au + b$, for $a \in \Fq^\times$ and $b \in \Fq$. In fact, it is known that the automorphisms of $F$ are of the form
    $$
    \varphi_{a,b,c,d}:
    \begin{cases}
        y \mapsto ay + c, \\
        z \mapsto a^2z + 2a\varepsilon y + d,
    \end{cases}
    $$
    where $a$ and $b$ are as above, and $c$ and $d$ satisfy $c^q-c = b$ and $d^q - d = \varepsilon a^2$.

    If $\varepsilon \not\in \Fq$ then it is less clear which automorphisms we get besides from the $\Fq$- translations in $y$ and $z$, and the automorphisms obtained from lifting $u \mapsto au$. We mention just one more special case, namely the case where $\varepsilon^q + \varepsilon = 0$. For such $\varepsilon$, define $\tilde{z} := \varepsilon z + (\varepsilon y)^2$ and $\tilde{y} = \xi y$, for some $\xi \in K$ such that $\xi^{q+1} = -\varepsilon$. A direct check then shows that $F = K(\tilde{y}, \tilde{z})$ with $\tilde{z}^q + \tilde{z} = \tilde{y}^{q+1}$, i.e., $F$ is the Hermitian function field. In particular the automorphism group is isomorphic to $\mathrm{PGU}(3,q)$, which has order $q^3(q^2-1)(q^3+1)$.
\end{example}

\begin{remark}
    The case $p = 2$ can be treated in a slightly more general setting: Suppose just that $u$ and $v$ have the same unique pole in $K(u,v)$ and that $f(u,v)$ is an irreducible polynomial of total degree $2$. Then, the valuation of the right hand side of Equation \ref{eq_F_mu_i} is at least $-2$, so it follows from \cite[Proposition 3.7.8]{Sti} and the discussion after Lemma \ref{lem_one_eq_irr} that $g(F) = 0$. 
\end{remark}

When $p$ is odd, the situation is a bit more complicated. In the above example we had $g(F) = \frac{q^2-q}{2}$, but the next example shows that this is not the only possible genus.

\begin{example}[One non-rational point at infinity]\label{ex_1_non_rational}
    Suppose $p\neq 2$ and $f(u,v) = (\varepsilon u - v)^2 - v$ for some $\varepsilon \in K \setminus \Fq$. The conic given by this equation has one non-$\Fq$-rational point at infinity. Let $P_\infty$ be corresponding place in $K(u,v)$, i.e., the unique pole of $u$ and $v$. Note that 
    $$
        v_{P_\infty} (u) = -2 \ \text{ and } \ v_{P_\infty}(v) = -2,
    $$
    while $v_{P_\infty}(\varepsilon u - v) = -1$. The $p$-rank of $F$ is again zero by Lemma \ref{lem_p_rank_0_one_pole}, and we can determine the genus as in the previous example. The right hand side of Equation \ref{eq_F_mu_i} has valuation $-2$ unless 
    $$
        \frac{\mu_i^{q} - \mu_i}{(\mu_i\eta)^{q} - \mu_i\eta} = -\varepsilon,
    $$
    but a direct check shows that the left hand side is in $\Fq$, so this never happens. We get that 
    $$
        g(F_{\mu_i}) = \frac{p-1}{2},
    $$
    and hence 
    $$
        g(F) = \frac{q^2-1}{p-1} \cdot \frac{p-1}{2} = \frac{q^2-1}{2}.
    $$

    Initial calculations with Magma suggest that the number of $\F_{q^{2q}}$-rational places is $q^{2q} + 1$ and that the function field is $\F_{q^{4q}}$-minimal, but we leave this for future work. As a final remark, we note that writing $t = y - z/\varepsilon^q$ and $s =  u - v/\varepsilon$ gives $F = K(s,t)$ with
    $$
        t^{q^2} - t = s^q - s + \alpha s^2, 
    $$
    where $\alpha = (\varepsilon^{q-1}-1)/\varepsilon^{q-2}$.
\end{example}

\begin{remark}
    The examples in this section cover all the possible genera one can obtain when starting from an irreducible conic. However, the isomorphism classes among the function fields with the same genus have yet to be described.
\end{remark}

\section{Three ordinary function fields with many automorphisms}\label{sec:three_Zieve}
In this section, we study three ordinary function fields with large automorphism groups, related to a curve first described by Zieve. The first two can be seen as double Artin-Schreier extensions of a rational function field while the last one is a triple Artin-Schreier extension.

\subsection{The Zieve function field}\label{sec_Zieve_curve}

For a prime power $q = p^n$, consider the function field $Z_q := K(x,y,z)$ defined by the equations

\begin{align*}
        y^q - y &= \frac{1}{x^q - x}, \\
        z^q - z &= \frac{x^{q+1}}{x^q - x}.
\end{align*}

This is a double Artin-Schreier extension of the projective line, and we could in principle use the methods of the previous sections to determine the genus and the $p$-rank, but in this case a more direct approach is easier (see Remark \ref{rmk_Zieve_plane_model}). 

In the extension $K(x,y)/K(x)$ there are exactly $q$ ramified places, namely the zeros of $(x-\lambda)$ for $\lambda \in \Fq$, and all of these are totally ramified. Hence a direct calculation, using the Hurwitz formula and the Deuring-Shafarevich formula, shows that 
$$
    g(K(x,y)) = \gamma(K(x,y)) = (q-1)^2, 
$$

where $\gamma(K(x,y))$ is the $p$-rank. \newline

In the extension $K(x,y,z)/K(x,y)$ the $q$ poles of $x$ are all totally ramified. We claim that these are the only ramified places. The other poles of $\frac{x^{q+1}}{x^q - x}$, and hence the only other potentially ramified places, are the zeros of $x^{q-1} - 1$. For a fixed $\lambda\in \Fq^\times$, denote by $P_\lambda$ the zero of $(x-\lambda)$ in $K(x,y)$. Since
$$
    \frac{x^{q+1}}{x^q - 1} = x^{q+1} (y^q - y) = \lambda^{q+1} (y^q - y) + (x^{q+1} - \lambda^{q+1})(y^q - y),   
$$
and 
$$
    (x^{q+1} - \lambda^{q+1})(y^q - y) = \frac{x^{q+1} - \lambda^{q+1}}{x^q - x} 
        = \frac{x^{q+1} - \lambda^{q+1}}{(x-\lambda)\prod_{\varepsilon\in \Fq\setminus \{ \lambda \}} (x-\varepsilon)}, 
$$
which has valuation $0$ at $P_\lambda$ (the numerator, when viewed as a polynomial, is separable and has $\gamma$ as a root), we get
$$
    v_{P_\lambda}\left(  \frac{x^{q+1}}{x^q - 1} - ((y\lambda^2)^q -y\lambda^2) \right) = 0.
$$
This shows that $P_\lambda$ splits in $\Fq(x,y,z)/\Fq(x,y)$, as claimed, and a direct calculation yields
$$
    g(K(x,y,z)) = \gamma(K(x,y,z)) = q^3 - q^2 - q + 1,
$$
i.e., the function field $Z_q$ is ordinary.

\begin{remark}\label{rmk_Zieve_plane_model}
    As mentioned previously, the genus and $p$-rank could also have been determined with the methods described in Section \ref{sec_preliminary}, but even the first step of finding the relation between $u := 1/(x^q-x)$ and $v := x^{q+1}/(x^q-x)$ is non-trivial. For $p = 2$ we have
    $$
        u^q v + uv^q + \mathrm{Tr}_{\Fq/\F_2}(uv) + 1 = 0,
    $$
    and irreducibility follows from $[K(x):K(u)] = q$ and $K(x) = K(u,v)$, which in turn follows from using the identities $x = x^q + 1/u$ and $x^2 = (x+v)/u$. The corresponding curve has $q+1$ non-singular points at infinity, so we could proceed without too much trouble. However, for $q$ odd we are less lucky. One can check that
    \begin{align*}
        u^qv + uv^q =
        \begin{cases}
            \hfil (2uv + 1) \cdot \prod_{i=1}^{(q-1)/4} \left( 1  + 4 uv - (\alpha_i + \alpha_i^q)^2 (uv)^2\right) \ &\text{ for } q \equiv 1 \mod 4, \text{ and } \\
            & \\
            \hfill 2 \cdot \prod_{i=1}^{(q+1)/4} \left( 1  + 4 uv - (\alpha_i + \alpha_i^q)^2 (uv)^2\right) \ &\text{ for } q \equiv 3 \mod 4,
        \end{cases}
    \end{align*}
    where the $\alpha_i$'s are elements of $\Fqq$ chosen in such a way that $\cup_{i} \left\{ \alpha_i, - \alpha_i, 1/\alpha_i, -1/\alpha_i\right\}$ is the set of all roots of $T^{q+1} + 1 \in \Fqq[T]$, except those that are also roots of $T^2 + 1$. 
    
    Again, irreducibility follows from $K(x) = K(u,v)$, which holds since $ux^2 + x - v = 0$ and $[K(x):K(u,v)]$ divides $[K(x) : K(u)] = q$, so the extension degree divides $\gcd(2,q) = 1$. Besides from the more complicated looking expression, we also find that the corresponding plane curve has singular points at infinity, so the situation is more complicated than for $q$ even.
\end{remark}

We now turn our attention to the automorphisms of $Z_q$. There is the usual elementary abelian subgroup of automorphisms consisting of the $\Fq$-translations in $y$ and $z$, but we also find automorphisms of the form 
$$
    \delta_\lambda : (x,y,z) \mapsto (\lambda x, \lambda^{-1}y, \lambda z),
$$
for $\lambda \in \Fq^\times$, and
$$
    \pi : (x,y,z) \mapsto (1/x, z, y).
$$
This means that the full automorphism group of $Z_q$ has order at least $2q^2(q-1)$, i.e., its size is at least of the same order of magnitude as $g(Z_q)$.

For $q = 2^n$, with $n\in \mN$, it was noticed by Zieve that $Z_q$ has a very large automorphism group compared to its genus. In fact, one finds additional automorphisms of the type
$$
    \gamma_{\mu,\nu} : (x,y,z) \mapsto (x + \mu, y, z + \mu^2 y + \nu)
$$
where $\mu \in \Fq$ and $\nu \in \Fqq$ with $\nu^q + \nu = \mu$. We will see in Section \ref{sec_aut_full}, except possibly for small values of $n$, that the full automorphism group of $Z_{2^n}$ is generated by the automorphisms mentioned above, and that it is isomorphic to the semidirect product of an elementary abelian group of order $q^2$ and a group isomorphic to $\PGL(2,q)$. In particular, for $q = 2^n$ and $n > 3$, we will see that the order of $\aut(Z_q)$ is $q^3(q^2-1)$, which is of the same order of magnitude as $g(Z_q)^{5/3}$. \newline

It seems natural to next look for a double Artin-Schreier extension of the projective line where all of $\PGL(2,q)$ can be lifted, also for odd $q$, but we have not been successful in this search. However, we will give an example of a double Artin-Schreier extension where a subgroup of order $q(q-1)$ can be lifted, and finally we will be able to lift all of $\PGL(2,q)$ in a triple Artin-Schreier extension for odd $q$.

\subsection{The modified Zieve function field}
Suppose $q$ is a power of an odd prime, and consider the function field $Z'_q := K(x,y,z)$ defined by the equations
\begin{align*}
    y^q - y &= \frac{1}{x^q - x}, \\
    z^q - z &=  \frac{x^q + x}{x^q-x}.
\end{align*}
We note that $u := 1/(x^q-x)$ and $v := (x^q+x)/(x^q -x)$ satisfies
$$
    u^{q-1}v - v^q + u^{q-1} + 1 = 0,
$$
and the polynomial $f(u,v) := u^{q-1}v - v^q + u^{q-1} + 1$ is irreducible. The corresponding curve has $q$ points at infinity and all of them are non-singular. This means that the poles of $u$ and $v$ are all simple in the rational function field $K(x) = K(u,v)$, so it follows from a similar argument as in the proof of Lemma \ref{lem_two_points_ordinary} that $Z'_q$ is ordinary. The genus can be determined either directly as for the Zieve function field or by means of the method described in Section \ref{sec_preliminary}. Using the latter approach, we find that
$$
    g((Z_q')_{\mu_i}) = 
    \begin{cases}
        (p-1)(q-1) & \text{ if } \mu_i \eta \in \Fq, \text{ and },\\
        (p-1)(q-2) & \text{ else,}
    \end{cases}
$$
so that
\begin{align*}
    g(Z'_q) &= \sum_{i=1}^m g((Z_q')_{\mu_i})\\
        &= \left( \frac{q^2-1}{p-1} - \frac{q-1}{p-1}\right) \cdot (p-1)(q-2) +  \left( \frac{q-1}{p-1}\right) \cdot (p-1)(q-1) \\
        &= (q-1)(q^2 - q -1).
\end{align*}

The automorphisms we see immediately are the usual $\F_q$-translations in $y$ and $z$, as well as those of the form
$$
    \gamma_\mu: (x,y,z) \mapsto (x + \mu , y, z + 2\mu y), 
$$
for $\mu \in \Fq$, and 
$$
    \delta_\lambda: (x,y,z) \mapsto (\lambda x, \lambda^{-1}y,z),
$$
for $\lambda \in \Fq^\times$. 
These automorphisms together generate a group of order $q^3(q-1)$, so $Z_q'$ is ordinary with an automorphism group that is at least of the same order of magnitude as $g(Z_q')^{4/3}$.

\subsection{The extended Zieve function field}\label{sec_spicy_curve}

Let $q = p^n$ for some odd prime $p$ and $n \in \mN$. In this case, the automorphism group of $Z_q$ does not seem to be remarkably large, but there is an extension field $\widetilde{Z}_q \supseteq Z_q$, which is ordinary and has a large automorphism group compared to its genus. In fact, consider $\widetilde{Z}_q := K(x,y,z,t)$ defined by the equations 
\begin{align*}
    y^q - y = \frac{1}{x^q - x}, \\
    z^q - z = \frac{x^{q+1}}{x^q - x},\\
    t^q - t = \frac{x^q + x}{x^q-x}.
\end{align*}

The extension $\widetilde{Z}_q / Z_q$ is unramified, so we cannot get irreducibility directly like for $Z_q$. Instead, irreducibility follows from considering first the extension $K(x,t)/K(x)$, then the extension $K(x,t,y)/K(x,t)$, and finally $K(x,t,y,z)/K(x,t,y)$. All of these are Artin-Schreier extensions with at least one totally ramified place: In the first extension the zero of $x-\lambda$ is totally ramified for $\lambda \in \Fq^\times$, in the second extension the $q$ zeros of $x$ are totally ramified, and in the last extension the $q^2$ poles of $x$ are totally ramified. Straightforward calculations using the Hurwitz formula and the Deuring-Shafarevich formula shows that 
$$
    g(K(x,y,z,t)) = \gamma(K(x,y,z,t)) = q^4 - q^3 - q^2 + 1,
$$
and in particular $\tZ_q$ is ordinary.

The first observation regarding the automorphism group of $\tZ_q$ is that it contains an elementary abelian subgroup of order $q^3$, namely the $\Fq$-translations of $y$, $z$ and $t$. Besides, we find automorphisms of the type:
$$
    \gamma_{\mu} : (x,y,z,t) \mapsto (x + \mu, y, z + \mu^2 y + \mu t, t + 2\mu y),
$$
for $\mu \in \Fq$, 
$$
    \delta_\lambda : (x,y,z,t) \mapsto (\lambda x, \lambda^{-1}y, \lambda z,t),
$$
for $\lambda \in \Fq^\times$, and finally the involution
$$
    \pi : (x,y,z,t) \mapsto (-1/x, z, y,-t).
$$

We will see in Section \ref{sec_aut_full} that these generate an automorphism group isomorphic to the semidirect product of an elementary abelian group of order $q^3$ and a group isomorphic to $\PGL(2,q)$, and that this is the full automorphism group, except possibly for small values of $q$. In particular, the order of the automorphism group is $q^4(q^2-1)$, i.e., at least of the same order of magnitude as $g(\tZ_q)^{3/2}$.

\section{Full automorphism groups of selected curves}\label{sec_aut_full}

In this section, we determine the full automorphism group of several of the function fields considered in the previous sections, namely for the Singer, Zieve, and extended Zieve function fields.

\subsection{Automorphisms of the Singer function field}
Let $q$ be a prime power, fix an algebraic closure $K$ of $\Fq$, and consider the function field $F := K(s,t)$ defined by 
$$
    (s^q-t)(t^q-s) = 1.
$$
In Example \ref{ex_singer}, we saw that this is an ordinary function field of genus $q^2-1$, and that the automorphism group has order at least $2q^2(q+1)$. In the following we will prove that the full automorphism group is no larger than this, for $q > 7$.

Let $x = t^q - s$ and recall that $F$ can be realized as a double Artin-Schreier extension over $K(x)$ with two totally ramified places. As a defining equation one could take either 
$$
    t^{q^2} - t = \frac{x^{q+1}+1}{x}  
$$
or 
$$
    s^{q^2}-s = \frac{x^{q+1}+1}{x^q}.
$$

The two totally ramified places $P_1$ and $P_2$ are, respectively, the zero and the pole of $x$ in $K(x)$. We denote the unique place in $F$ above $P_1$ by $Q_1$, and $Q_2$ is the unique place above $P_2$. From the defining equations we see that 
$$
    (s)_\infty  = q Q_2 + Q_1
$$
and 
$$
    (t)_\infty = qQ_1 + Q_2,
$$
and the divisor of $x$ in $F$ is $q^2(Q_2-Q_1)$. These observations will be useful later.

Recall that 
\begin{align*}
    E &:= \left\{ \sigma_{\alpha,\beta} : (s,t) \mapsto (s + \alpha - \beta \sqrt{\varepsilon}, \ t + \alpha + \beta \sqrt{\varepsilon}) \right\} \simeq E_{q^2} \\
    H_1 &:= \{ \delta_{\sigma}: (s,t) \mapsto (\lambda s, \lambda^{-1} t) \mid \lambda^{q+1} = 1, \lambda \in \Fqq\} \simeq C_{q+1} \\
    H_2 &:= \langle \pi: (s,t) \mapsto (t,s) \rangle 
\end{align*}
are subgroups of $\aut(F)$ for $q$ odd. For $q$ even, the same is true for $H_1$ and $H_2$ as above but the elementary abelian subgroup is instead
$$
    E := \left\{ \sigma_{\alpha,\beta} (s,t) \mapsto (s + \alpha + \beta \xi, \ t + \alpha + \beta (\xi+1))\right\} \simeq E_{q^2},
$$

where $\xi$ is as described in Remark \ref{remark_even_singer}. 

In any case, it is clear that $H := \langle H_1,H_2\rangle$ is a dihedral group of order $2(q+1)$, and a direct check shows that $H$ normalizes $E$. We will show that $\aut(F)$ is, in most cases, no larger than $G := E \rtimes H$.

\begin{theorem}\label{thm_full_aut_singer}
    For $q > 7$ and $F = K(s,t)$ as defined above we have 
    $$
        \aut(F) \simeq E_{q^2} \rtimes D_{2(q+1)}.
    $$
    In particular, the order of the automorphism group is $q^2(q+1)$, so $|\aut(F)| \sim g(F)^{3/2}$.
\end{theorem}

The key to the proof is to consider certain Riemann-Roch spaces, as in the proof of \cite[Proposition 3.1]{Korchmaros_Montanucci_2019}, and we thank Maria Montanucci for suggesting this strategy. First, we consider the orbit containing $Q_1$ and $Q_2$.

\begin{lemma}\label{lem_orbit_Singer}
    If $q > 7$ then $\Omega := \{Q_1,Q_2\}$ is an orbit under $\aut(F)$.
\end{lemma}

\begin{proof}
    The automorphism $\pi$ switches $Q_1$ and $Q_2$, so they are in fact in the same $\aut(F)$-orbit. Denote this orbit by $\mathcal{O}$, and suppose exists a place $Q\in \mathcal{O}\setminus \Omega$. 

    We know that the $q^2$ automorphisms of $E$ fix both $Q_1$ and $Q_2$, and so do the $q+1$ automorphisms of $H_1$. This means that the stabilizer of, say, $Q_1$ in $\aut(F)$ has order at least $q^2(q+1)$. Moreover, the fixed field of $E$ is $K(x)$, and the only ramified places in $F/K(x)$ are $P_1$ and $P_2$, so $E$ acts with long orbits outside of $\Omega$. Since $H_1$ acts on $K(x)$ by mapping $x$ to $\lambda x $, for $\lambda \in \Fqq$ with $\lambda^{q+1}=1$, we also find that $H_1$ acts with long orbits outside of $\Omega$. This means that
    $$
        |\mathcal{O}| = 2 + k\cdot q^2(q+1),
    $$
    for some $k \in \mZ_{\geq 0}$. 

    If $k = 0$ we are done. Otherwise, the orbit-stabilizer theorem combined with the observations above tells us that 
    $$
        |\aut(F)| \geq q^2(q+1)(2 + q^2(q+1)).
    $$
    On the other hand, $F$ is ordinary of genus $g := q^2-1$ so by \cite[Theorem 3]{nakajima_p-ranks_1987} we have
    $$
        |\aut(F)| \leq 84 g (g-1) = 84(q^2-1)(q^2-2).
    $$
    Combining the inequalities yields a contradiction unless $q \leq 7$, so the claim follows.
    
    
\end{proof}

For the next lemma, we need some information on the Weierstrass semigroup at $Q_1$ and $Q_2$. In fact, it is known that an Artin-Schreier extension with exactly two ramified places, both totally ramified, is classical and that the two ramified places have ordinary gapsequences, i.e., the gaps at $Q_1$ and $Q_2$ are exactly $1, 2, \dots , q^2 - 1$ (see, e.g., \cite{Garcia_1989}). In our case, these gaps can be found by considering (regular) differentials of the form
$$
    \omega_{k,l} := s^kt^l \frac{dx}{x},
$$
for $k,l \in \{0,1, \dots, q-1\}^2 \setminus(q-1,q-1)$. The fact we will need, which follows from these observations, is that the Riemann-Roch spaces $\mathcal{L}(qQ_1)$ and $\mathcal{L}(qQ_2)$ have dimension one, i.e., they are simply equal to $K$.

\begin{lemma}\label{lem_stab_Singer}
    If $\Omega = \{Q_1,Q_2\}$ is an orbit under $\aut(F)$, then the stabilizer of $Q_1$ in $\aut(F)$, which also fixes $Q_2$, is isomorphic to $E_{q^2} \rtimes C_{q+1}$. More precisely, it can be written as $E\rtimes H_1$, with notation as above. 
\end{lemma}

\begin{proof}
    Suppose $\Omega$ is an orbit and denote by $A_{Q_1}$ the stabilizer of $Q_1$ in $\aut(F)$. By \cite[Theorem 11.49]{HKT} we may write
    $$
        A_{Q_1} = E' \rtimes C, 
    $$
    where $E'$ is a $p$-Sylow subgroup of $A_{Q_1}$ and $C$ is a cyclic subgroup of order coprime to $p$. We know that $E \subseteq A_{Q_1}$, so $E \subseteq E'$, and after possibly replacing $C$ one of its conjugates we may assume that $H_1 \subseteq C$. The claim is now that $E' = E$ and $C = H_1$.

    The first part follows from \cite[Theorem 1 (i)]{nakajima_p-ranks_1987} since $F$ is ordinary. In fact, we have
    $$
        |E'| \leq \frac{p}{p-2}(q^2-2) < pq^2,
    $$
    for $q$ odd, and 
    $$
        |H_2|\cdot|E'| \leq 4(q^2-1) < 2(2q^2),
    $$
    for $q$ even, so $E' = E$ in any case. 

    To show the second part of the claim we consider the Riemann-Roch space $\mathcal{L}(qQ_2 + Q_1)$. It is clear from our previous observations that 
    $$
        \dim \mathcal{L}(qQ_2 + Q_1) \leq \dim \mathcal{L}(qQ_2) + 1 = 2.
    $$
    Since $s \in \mathcal{L}(qQ_2 + Q_1)\setminus K$ the dimension is in fact $2$, and we may write
    $$
        \mathcal{L}(qQ_2 + Q_1) = K \oplus sK.
    $$
    Now any element, and in particular a fixed generator $\alpha \in C$, fixes both $Q_1$ and $Q_2$ and hence acts on $\mathcal{L}(qQ_2 + Q_1)$. This means that there exists $\mu,\lambda \in K$ such that
    $$
        \alpha(s) = \mu  + \lambda s.
    $$
    The order of $\alpha$ is coprime to $p$ so $\mu = 0$, i.e., 
    $$
        \alpha(s) = \lambda s.
    $$
    By considering $\mathcal{L}(qQ_1 + Q_2)$ we similarly get that 
    $$
        \alpha(t) = \lambda't,
    $$
    for some $\lambda' \in K$, and 
    $$
        \alpha(x) = \lambda'' x,
    $$
    for some $\lambda'' \in K$, since $\alpha$ fixes the divisor of $x$. By combining the above and using the defining equation for $F$ we find 
    $$
        \lambda''(t^q-s) = \alpha(x) = \alpha(t^q-s) = \lambda'^qt^q - \lambda s,
    $$
    and
    $$
        (\lambda'')^{-1}(s^q-t) =  \alpha(1/x) = \alpha (s^q-t) = \lambda^qs^q - t^q.
    $$
    The first equation implies $\lambda = \lambda''$ and it follows from the second equation that $\lambda^q = (\lambda'')^{-1}$ and $\lambda' = (\lambda'')^{-1}$. We conclude that $\lambda^{q+1} = 1$ and $\lambda' = \lambda^{-1}$, so $\alpha$ is an element of $H_1$, i.e., $C=H_1$ as wished.
\end{proof}

The proof of Theorem \ref{thm_full_aut_singer} now follows from combining the two lemmas.

\begin{proof}[Proof of Theorem \ref{thm_full_aut_singer}.]
    Write $A$ for $\aut(F)$ and suppose $q > 7$. By Lemma \ref{lem_orbit_Singer} we have that $\Omega = \{Q_1,Q_2\}$ is an $A$-orbit and it then follows from Lemma \ref{lem_stab_Singer} that the stabilizer of $Q_1$ is isomorphic to $E_{q^2} \rtimes C_{q+1}$. The orbit-stabilizer theorem implies
    $$
        |A| = |\Omega| \cdot |A_{Q_1}| = 2q^2(q+1),
    $$
    so we conclude that $A = \aut(F)$ is nothing more than $G = E\rtimes H$, as claimed.
\end{proof}

\begin{remark}
    For $q = 2$, computations with a computer using, e.g., Magma shows that the automorphism group contains the unique simple group of order 168 (while in this case $|G| = 2^3(2+1) = 24$). In fact $F$ is, in this case, isomorphic to the function field of the Klein quartic, so the group of order 168 is the full automorphism group.
    
    We leave it as an open problem to determine (by means of theoretical or computational arguments) the full automorphism group for $q \in \{ 3,4,5,7\}$.
\end{remark}

\subsection{Automorphisms of the Zieve function field}

For $q = 2^n$, with $n \in \mN$, we claimed in Section \ref{sec_Zieve_curve} that the ordinary function field $Z_q$ has an automorphism group isomorphic to the semidirect product of an elementary abelian group of order $q^2$ and $\PGL(2,q)$. We will show that this is true and that this group is the full automorphism group, except possibly for small values of $q$:

\begin{theorem}\label{thm_aut_even}
For $q > 8$, we have
$$
    \aut(Z_q) \simeq E_{q^2} \rtimes \PGL(2,q).
$$
In particular, the order of the automorphism group is $q^3(q^2-1)$, so $|\aut(Z_q)| \sim g(Z_q)^{5/3}$.
\end{theorem}

To show the theorem, first recall that
\begin{align*}
    E &:= \{ \sigma_{\alpha, \beta}: (x,y,z) \mapsto (x, y + \alpha, z + \beta) \mid \alpha, \beta \in \Fq \} \simeq E_{q^2}, \\
    H_1 &:= \{\gamma_{\mu,\nu} : (x,y,z) \mapsto (x + \mu, y, z + \mu^2 y + \nu) \mid \mu \in \Fq, \nu^q + \nu = \mu \} \simeq E_{q^2}, \\
    H_2 &:= \{ \delta_\lambda : (x,y,z) \mapsto (\lambda x, \lambda^{-1}y, \lambda z)\mid \lambda \in \Fq^\times \} \simeq C_{q-1} \text{ and }\\
    H_3 &:= \langle \pi : (x,y,z) \mapsto (1/x, z, y) \rangle,   
\end{align*}

are subgroups of $\aut(Z_q)$. Moreover, if we define 
$$
    G := \langle E, H_1, H_2, H_3 \rangle \subseteq \aut(Z_q),
$$
then a direct check shows that $E$ is normal in $G$. We will determine the structure of $G$ and then show that $\aut(Z_q)$ contains nothing more than $G$.

\begin{lemma}\label{lem_G_aut_even}
    We have 
    $$
        G = E \rtimes H,
    $$
    for a subgroup $H\subseteq G$, with $H \simeq \PGL(2,q)$.
\end{lemma}

\begin{proof}
    Choose $\nu_0$ such that $\nu_0^q+\nu_0=1$. For each $\mu \in \mathbb{F}_q$, we have  $\gamma_{\mu, \mu \nu_0} \in H_1$ because
    $$
        \left(\mu \nu_0\right)^q+\mu \nu_0=\mu\left(\nu_0^q+\nu_0\right)=\mu,
    $$
    and one can easily check that $H_1':=\left\{\gamma_{\mu, \mu \nu_0} \mid \mu \in \mathbb{F}_q\right\}$ is a subgroup of $H_1$.

    Define $H := \langle \pi, H_1' \rangle$. We will show that $H \simeq \PGL(2,q)$, and that $H$ intersects trivially with $E$. In accordance with the notation of \cite{Dickson}, we write $T := \pi$, and $S_\mu := \gamma_{\mu,\mu\nu_0}$. Then, $H$ is generated by $T$ and $\{S_\mu\}_{\mu\in\Fq}$, and it can be checked directly that the following relations hold: 
    
    \vspace{0.2cm}
    \begin{enumerate}[label=(\roman*)]
        \item $S_0 = T^2 = \mathrm{Id}$,
        \item $S_\mu S_{\lambda} = S_{\mu + \lambda}$, for $\mu,\lambda \in \Fq$, and
        \item $S_\lambda T S_\mu T S_{\frac{\lambda-1}{\lambda\mu -1}} T S_{-(\lambda \mu - 1)} T S_\frac{\mu-1}{\lambda\mu -1} T = \mathrm{Id}$.
    \end{enumerate}
    \vspace{0.2cm}

    By \cite[Theorem 278]{Dickson} and its corollary, these are the relations for $\mathrm{PSL}(2,q)$. Since $q$ is even we have $\mathrm{PSL}(2,q)\simeq \PGL(2,q)$, so $H$ is isomorphic to a quotient of $\mathrm{PGL}(2,q)$. 

    On the other hand, the projection $G \to G/E$ induces an action of $H$ on $\Fq(x)$, and from the explicit description of the automorphisms, we see that $H$ in this way acts as $\PGL(2,q)$ on $\Fq(x)$. The kernel of the action is exactly $H \cap E$, since any $\varphi \in H$ which fixes $x$ also fixes $y^q - y$ and $z^q - z$, and hence is in $E$. By combining the above we get both that $H \simeq \PGL(2,q)$ and that the intersection of $E$ and $H$ is trivial, so $E \rtimes H$ is a subgroup of $G$ of order $q^3(q^2+1)$.

    Finally, consider the action of $G/E$ on $\Fq(x)$. By a similar argument as above, we get that $G/E$ acts faithfully as $\PGL(2,q)$ on $\Fq(x)$. It follows that $G/E \simeq \PGL(2,q)$, and in particular
    $$
        |G| = q^3(q^2-1),
    $$
    so $G = E \rtimes H$ as wished.
\end{proof}

Proving Theorem \ref{thm_aut_even} now amounts to showing that $\aut(Z_q)$ has no more elements than those in $G$. To this end, we study the short orbits of the action of $G$ on the places of the function field $K(x,y,z)$. \newline

Denote by $\Omega_1$ the set of places lying above the $\F_q$-rational places of $K(x)$, and let $\Omega_2$ be the set of places lying above the places of $K(x)$ that are $\Fqq$-rational but not $\Fq$-rational. We know, e.g., from \cite[Theorem 11.92]{HKT}, that the action of $\PGL(2,q)$ on $K(x)$ gives rise to exactly two short orbits. One is of order $q+1$ and consists of all the $\F_q$-rational places and the other one is of order $q^2-q$ and consists of all the $\F_{q^2}$-rational places that are not $\F_q$-rational. Moreover, the extension $K(x,y,z)/K(x)$ is Galois, with Galois group $E \subseteq G$, and the ramified places are in $\Omega_1$, so we conclude that the short orbits of $G$ are exactly $\Omega_1$ and $\Omega_2$. \newline

Note that $|\Omega_1| = q(q+1)$, $|\Omega_2| = q^2(q^2-q)$, and $|G| = q^3(q^2-1)$, so the stabilizer, $G_P$, of a place $P \in \Omega_1$ has order $q^2(q-1)$ and the stabilizer, $G_Q$, of a place $Q\in\Omega_2$ has order $(q+1)$. We will show that $\Omega_1$ is also an orbit under the full automorphism group, except possibly for small values of $q$:

\begin{lemma}\label{lem_orbit_even}
    If $q > 8$, then $\Omega_1$ is an orbit under $\aut(Z_q)$.
\end{lemma}

\begin{proof}
We write $A := \aut(Z_q)$ for convenience. Consider some $P\in \Omega_1$, and denote by $\mathcal{O}_P$ the $A$-orbit of $P$. Suppose for contradiction that $\mathcal{O}_P$ contains some place $Q$ not in $\Omega_1$. If $Q$ is in $\Omega_2$, then the stabilizer of $Q$ in $A$, $A_Q$, has order at least $\mathrm{lcm}(G_P,G_Q) = q^2(q^2-1)$, so the orbit-stabilizer theorem implies
\begin{align*}
    |A| &= |\mathcal{O}_P| \cdot |A_P| \\ 
        &\geq (|\Omega_1| + |\Omega_2|)(q^2(q^2-1)) \\
        &= (q^4-q^3+q^2+q)(q^4-q^2) \\
        &= q^8 - q^7 + 2q^5-q^4-q^3.
\end{align*}

If instead $Q \in \mathcal{O}_P$ is neither in $\Omega_1$ nor in $\Omega_2$, then $Q$ is in a long $G$-orbit, so we get 
\begin{align*}
    |A| &= |\mathcal{O}_P| \cdot |A_P| \\ 
        &\geq (|\Omega_1| + |G|)(q^2(q-1)) \\
        &= (q^5-q^3+q^2+q)(q^3-q^2) \\
        &= q^8 - q^7 -q^6 + 2q^5-q^3.
\end{align*}

Hence, in any case, if $\mathcal{O}_P$ is strictly larger than $\Omega_1$ then we have
$$
    |A| \geq q^8 - q^7 -q^6 + 2q^5-q^3.
$$

On the other hand, since $Z_q$ is ordinary and of genus $g = q^3 - q^2 - q + 1 \geq 3$, we have 
$$
    |A| \leq 84g(g-1) = 84(q^6 - 2q^5 - q^4 +3q^3 -q).
$$

Combining the two inequalities yields a contradiction for $q > 8$, and the claim follows.


\end{proof}

To finish the proof of Theorem \ref{thm_aut_even} it is now sufficient to show that the stabilizer of a place $P \in \Omega_1$ is no larger than $G_P$:

\begin{lemma}\label{lem_stabilizer_even}
     If $q>8$, then for $P \in \Omega_1$, the stabilizer of $P$ in $\aut(Z_q)$ is equal to the stabilizer of $P$ in $G$. In particular, the order of the stabilizer is $q^2(q-1)$. 
\end{lemma}

\begin{proof}
    Define $A := \aut(Z_q)$. We consider a particular place in $P \in \Omega_1$, namely the unique place of $Z_q = K(x,y,z)$ that is a pole of $x$ and $z$ and a zero of $y$. Note that the stabilizer of $P$ in $A$, $A_P$, contains both $H_1$ and $H_2$. By \cite[Theorem 11.49]{HKT} we can write 
    $$
        A_P = S \rtimes C,
    $$
    where $S$ is the unique $2$-Sylow subgroup of $A_P$ and $C$ is cyclic of odd order. All elements of $H_1$ are involutions, so $H_1$ is contained in $S$. Moreover, we may assume after perhaps conjugating, that $C$ contains $H_2$. We will show that in fact $S = H_1$ and $C = H_2$.
    
    Since $Z_q$ is ordinary, we know that the second ramification group is always trivial, so $S$ is finite abelian (see \cite[Theorem 2]{nakajima_p-ranks_1987} and \cite[Theorem 11.74]{HKT}). In particular, any subgroup of $S$ is normal. Consider the subgroup $H_1^0 := \{\gamma_{0,\nu} \mid \nu^q + \nu = 0\} \subseteq S$. By the above, $H_1^0$ is normal in $S$, and the fixed field of $H_1^0$ is  $K(x,y)$, so $S/H_1^0$ is isomorphic to a subgroup of $\aut(K(x,y))$. The automorphism group of this (generalized) Artin-Mumford curve is known (see \cite[Theorem 1.1] {MZ_gen_Artin-Mumford}). In particular, the only automorphisms of $K(x,y)$ of order 2 that fix the zero of $y$ are those of the form $x \mapsto x + \mu$, for $\mu \in \Fq$. It follows that $S/H_1^0$ has order at most $q$, so we must have $S = H_1$ and $|S| = q^2$.

    Next, we consider the cyclic group $C$. The fixed field of $S = H_1$ is $K(y)$, so $C \simeq A_P/S$ is isomorphic to a subgroup of $\aut(K(y))$. The zero of $y$ is fixed, so \cite[Theorem 11.91]{HKT} tells us that $A_P/S$ has exactly two short orbits, each of length one. The image of $H_2$ in $A_P/S$ gives rise to non-trivial automorphisms fixing the pole of $y$, so the pole of $y$ must be the other place fixed by $A_P/S$. Hence, the divisor of $y$ is fixed, so the only options are $y \mapsto \lambda y$, for $\lambda \in K^\times$. In fact, by Lemma \ref{lem_orbit_even}, $A_P$ acts on the $\F_q$-rational places of $K(x,y,z)$, so we must have $\alpha \in \Fq^\times$. This means $|C| = |A_P/S| \leq q-1$, so in fact $C=H_2$ and $|C| = q-1$.
    
    By combining the above we get that $A_P = G_P$ and $|A_P| = q^2(q-1)$.
\end{proof}

\begin{proof}[Proof of Theorem \ref{thm_aut_even}.]
    We write $A$ for $\aut(Z_q)$. Suppose $q> 8$ and consider some $P\in \Omega_1$. By Lemma \ref{lem_orbit_even}, Lemma \ref{lem_stabilizer_even}, and the orbit-stabilizer theorem we have 
    $$
        |A| = |\mathcal{O}_P| \cdot |A_P| = |\Omega_1| \cdot |G_P| = |G|.
    $$

    In particular, $A$ is equal to $G$, so the theorem follows from Lemma \ref{lem_G_aut_even}. 
\end{proof}

\begin{remark}
    For $q = 2$, computations using Magma show that the automorphism group contains the unique simple group of order 168 (while in this case $|G| = 2^3(2^2-1) = 24$). In fact, like in the previous example, $Z_2$ is isomorphic to the function field of the Klein quartic, so the group of order 168 is the full automorphism group.
    
    We leave it as an open problem to determine the full automorphism group in the cases $q=4$ and $q=8$. 
\end{remark}

\subsection{Automorphisms of the extended Zieve function field}

For $q = p^n$, with $n \in \mN$ and $p$ an odd prime, we claimed in Section \ref{sec_spicy_curve} that the ordinary function field $\widetilde{Z}_q$ has an automorphism group isomorphic to the semidirect product of an elementary abelian group of order $q^3$ and $\PGL(2,q)$. We will show that this is true, and that this group is the full automorphism group, except possibly for small values of $q$.

First, we recall that the automorphism group contains the following subgroups:
\begin{align*}
    E &:= \{ \sigma_{\alpha, \beta,\nu}: (x,y,z,t) \mapsto (x, y + \alpha, z + \beta, t + \nu) \mid \alpha, \beta, \nu \in \Fq \} \simeq E_{q^3}, \\
    H_1 &:= \{\gamma_{\mu} : (x,y,z,t) \mapsto (x + \mu, y, z + \mu^2 y + \mu t, t + 2\mu y) \mid \mu \in \Fq\} \simeq E_{q}, \\
    H_2 &:= \{ \delta_\lambda : (x,y,z,t) \mapsto (\lambda x, \lambda^{-1}y, \lambda z,t)\mid \lambda \in \Fq^\times \} \simeq C_{q-1} \text{ and }\\
    H_3 &:= \langle \pi : (x,y,z,t) \mapsto (-1/x, z, y,-t) \rangle \simeq C_2.
\end{align*}

Moreover, if we define 
$$
    G := \langle E, H_1, H_2, H_3 \rangle \subseteq \aut(\tZ_q),
$$

then a direct check shows that $E$ is normal in $G$. As with $Z_q$, we will determine the structure of $G$ and then show that $\aut(\tZ_q)$ contains nothing more than $G$. The goal is to prove the following theorem:

\begin{theorem}\label{thm_aut_odd}
    For $q > 37$, we have 
    $$
        \aut(\tZ_q) \simeq E_{q^3} \rtimes PGL(2,q).
    $$
    In particular, the order of the automorphism group is $q^4(q^2-1)$, so $|\aut(\tZ_q)| \sim g(\tZ_q)^{3/2}$.
\end{theorem}

We will again split the proof into several lemmas. The first lemma gives the structure of $G$.

\begin{lemma}\label{lem_G_aut_odd}
    We have 
    $$
        G = E \rtimes H,
    $$
    for a subgroup $H \subseteq G$, with $H \simeq \PGL(2,q)$.
\end{lemma}

\begin{proof}
    We claim that $H = \langle H_1,H_2,\pi\rangle$ works. First, we consider $H' := \langle H_1, \pi \rangle$ and show that $H' \simeq \PSL(2,q)$. To see this, we proceed as in the proof of Lemma \ref{lem_G_aut_even} and check that the relations given in the Corollary to \cite[Theorem 278]{Dickson} hold for $T := \pi$ and $S_\mu := \gamma_\mu$ for $\mu \in \Fq$. This implies that $H'$ is isomorphic to a quotient of $\PSL(2,q)$. 

    On the other hand, the projection $G \to G/E$ induces an action of $H'$ on $\Fq(x)$, and from the explicit description of the automorphisms, we see that $H'$ in this way acts as $\PSL(2,q)$ on $\Fq(x)$. The kernel of the action is exactly $H' \cap E$, since any $\varphi \in H$ which fixes $x$ also fixes $y^q - y$ and $z^q - z$, and hence is in $E$. By combining the above we get both that $H' \simeq \PSL(2,q)$ and that the intersection of $E$ and $H'$ is trivial.

    The next step is to show that $H = H' H_2$. We check directly that 
    $$
        (\delta_\lambda)^{-1} \circ \gamma_\mu \circ \delta_\lambda = \gamma_{\lambda\mu} \in H',
    $$
    and
    $$
        (\delta_\lambda)^{-1} \circ \pi \circ \delta_\lambda = \pi \circ \gamma_{-1/\lambda} \circ \pi \circ \gamma_{1-\lambda} \circ \pi \circ \gamma_1 \circ \pi \circ \gamma_{\frac{\lambda-1}{\lambda}}\circ \pi \in H',
    $$
    for any $\mu \in \Fq$ and $\lambda \in \Fq^\times$. It follows that $H'H_2$ is a group, and hence $H = H'H_2$.

    From the above we also get that
    \begin{align*}
        \delta_{\lambda^2} &= \pi \circ \delta_{\lambda^{-1}} \circ \pi \circ \delta_\lambda \\
        &= \gamma_{-1/\lambda} \circ \pi \circ \gamma_{1-\lambda} \circ \pi \circ \gamma_1 \circ \pi \circ \gamma_{\frac{\lambda-1}{\lambda}}\circ \pi \in H'
    \end{align*}

    for $\lambda \in \F_q^\times$, so all elements of $H_2$ coming from squares of $\F_q^\times$ are also in $H'$. Using this, we obtain
    
    $$
        |H|  = \frac{|H'|\cdot |H_2|}{|H'\cap H_2|} \leq 2 |H'|.
    $$

    Since $H' \simeq \PSL(2,q)$, we conclude that 
    $$
        |H| \leq 2|\PSL(2,q)| = |\PGL(2,q)|.
    $$

    On the other hand, like previously, the projection $G \to G/E$ induces an action of $H$ on $\Fq(x)$, and from the explicit description of the automorphisms, we see that $H$ acts as $\PGL(2,q)$ on $\Fq(x)$. The kernel of the action is exactly $H \cap E$ by the same argument as for $H'$. By combining this with the bound on the size of $H$ we get both that $H \simeq \PGL(2,q)$ and that $E$ and $H$ have trivial intersection, so $E \rtimes H$ is a subgroup of $G$ of order $q^4(q^2+1)$. 

    Finally, we repeat the argument from the end of the proof of Lemma \ref{lem_G_aut_even} to obtain $|G| = q^4(q^2+1)$, and hence $G = E \rtimes H$ as wished. 
\end{proof}

We proceed, like for $q$ even, to show that $A := \aut(\tZ_q)$ contains no more automorphisms than those in $G$. To this end, we study the short orbits of the action of $G$ on the places of $K(x,y,z,t)$. 

Denote by $\Omega_1$ the set of places lying above the $\F_q$-rational places of $K(x)$, and let $\Omega_2$ be the set of places lying above the places of $K(x)$ that are $\Fqq$-rational but not $\Fq$-rational. An identical argument to the one used for $q$ even shows that the short orbits of $G$ are exactly $\Omega_1$ and $\Omega_2$. We note that $|\Omega_1| = q^2(q+1)$, $|\Omega_2| = q^3(q^2-q)$, and $|G| = q^4(q^2-1)$, so the stabilizer in $G$ of a place in $\Omega_1$ has order $q^2(q-1)$, and for a place in $\Omega_2$ the order of the stabilizer is $q+1$. We will use these facts to show that $\Omega_1$ is also an orbit under the full automorphism group, except possibly for small values of $q$:

\begin{lemma}\label{lem_orbit_odd}
    If $q > 37$ then $\Omega_1$ is an orbit under $\aut(\tZ_1)$.
\end{lemma}

\begin{proof}
    We write $A  = \aut(Z_q)$ for brevity. Consider some $P \in \Omega_1$ and the corresponding $A$-orbti, $\mathcal{O}_P$. Suppose for contradiction that $\mathcal{O}_P$ contains a place, $Q$, which is not in $\Omega_1$. We distinguish between two different cases:

    Suppose $Q$ is also in $\Omega_2$. Then the order of the stabilizer of $P$ in $A$ has order divisible by both $q^2(q+1)$ and $q-1$, i.e., it is at least $q^2(q^2-1)/2$. Therefore, the orbit-stabilizer theorem implies
    
    \begin{align*}
        |A| &\geq (|\Omega_1| + |\Omega_2|)\cdot \mathrm{lcm}(|G_P|,|G_Q|) \\
        &= \frac{1}{2}(q^5-q^4 +q^3 + q^2)(q^2(q^2-1)) \\
        &= \frac{1}{2}(q^9 - q^8 + 2 q^6 - q^5 - q^4).
    \end{align*}

    If instead $Q$ is in $\mathcal{O}_P\setminus (\Omega_1 \cup \Omega_2)$, then the orbit-stabilizer theorem yields 
    
    \begin{align*}
        |A| &\geq (|\Omega_1| + |G|)\cdot |G_P| \\
        &= (q^6 - q^4 + q^3 + q^2)(q^2(q-1)) \\
        &= q^9 - q^8 - q^7 + 2 q^6 - q^4.
    \end{align*}

    On the other hand, since $\tZ_q$ is ordinary, $p$ is odd, and the genus $g = q^4 - q^3 - q^2 + 1$ is even, we have from \cite{Montanucci_Speziali_odd_char_2020} that
    $$
        |A| \leq 821.37 \cdot g^{7/4} = 821.37(q^4-q^3-q^2 + 1)^{7/4}.
    $$
    By comparing the inequalities, we obtain a contradiction for $q > 37$, and the claim follows. 

\end{proof}

Next, we turn our attention to the stabilizer in $A$ of a place in $\Omega_1$:

\begin{lemma}\label{lem_stabilizer_odd}
    If $q > 37$, then the stabilizer in $\aut(\tZ_q)$ of a place $P \in \Omega_1$ is equal to the stabilizer of $P$ in $G$. In particular, the order of the stabilizer is $q^2(q-1)$.
\end{lemma}

\begin{proof}
    Define $A := \aut(\widetilde{Z}_q)$ and choose $P\in \Omega_1$ among the $q$ places that are poles of $x$ and $z$, and zeros of $y$. We can write 
    $$
        A_P = S \rtimes C,
    $$
    where $S$ is the unique $p$-Sylow subgroup of $A_P$ and $C$ is cyclic of odd order. Again, $S$ is abelian since $\tZ_q$ is ordinary (see \cite[Theorem 2]{nakajima_p-ranks_1987} and \cite[Theorem 11.74]{HKT}), so in particular every subgroup is normal.
    
    Consider the group $E^0 := \{\sigma_{0,\beta,0} \mid \beta \in \Fq\} \subseteq E$. It is clear that $E^0$ is in $S$, and so is $H_1$. Moreover, one can check that $E^0 H_1 = H_1H_q^0$, so $S' := E^0 H_1$ is a subgroup of $S$ of order $q^2$. Since $S'$ is normal in $S$, we know that $S$ acts on the $S'$-orbits of the places of $\widetilde{Z}_q$. 
    
    From Lemma \ref{lem_orbit_odd} we know that $A$, and hence $S$, acts on the set of zeros and poles of $y$ in $\widetilde{Z}_q$, since this is exactly $\Omega_1$. By the above, this means that $S$ acts on the $S'$-orbits of the poles and zeros of $y$ in $\widetilde{Z}_q$. The $S'$-orbit of each zero of $y$ in $\widetilde{Z}_q$ has length one, while the $S'$-orbit of each pole of $y$ has length $q^2$. Hence, $S$ acts on the zeros and poles of $y$ separately, i.e., $S$ fixes the divisor of $y$. Since the elements of $S$ have order $p$, this means that $S$ fixes $y$. Any automorphism fixing $y$ is already in $G$, and we know that $|G_P| = q^2(q-1)$, so we conclude that $S' = S$. \newline 
    
    Next, we consider the cyclic group $C$. Now that we know $S$, which is normal in $A_P$, we can repeat the previous argument to see that $A_P$ fixes the divisor of $y$. This means that any automorphism in $A_P$ maps $y$ to $\lambda y$, for some $\lambda \in K$. Since $\Omega_1$ is an $A$-orbit, we must have $\lambda \in \F_q$, and this means that any $\varphi \in C$ satisfies $\varphi^{q-1}(y) = y$. Like previously, this implies $\varphi^{q-1}(x) = x + \lambda$, for some $\lambda \in \F_q$, but since the order of $\varphi^{q-1}$ is coprime with $q$ we must have $\lambda = 0$. A similar argument for $z$ and $t$ shows that $\varphi^{q-1}$ is the identity. Hence, $|C| = q-1$ and the claim follows.  
\end{proof}

Finally, we combine the above to obtain a proof of the main theorem of this section:

\begin{proof}[Proof of Theorem \ref{thm_aut_odd}.]
    We write $A$ for $\aut(\widetilde{Z}_q)$. Suppose $q> 37$, and consider some $P\in \Omega_1$. By applying the orbit-stabilizer theorem together with Lemma \ref{lem_orbit_odd} and Lemma \ref{lem_stabilizer_odd} we get 
    $$
        |A| = |\mathcal{O}_P| \cdot |A_P| = |\Omega_1| \cdot |G_P| = |G|.
    $$

    In particular, $A$ is equal to $G$ so the theorem follows from Lemma \ref{lem_G_aut_odd}. 
\end{proof}

\begin{remark}
    For $q\leq 37$, new arguments would be needed to determine the full automorphism group, and possibly prove that it still has size $q^4(q^2-1)$; we leave it as an open problem.
\end{remark}

\bibliographystyle{abbrv}
\bibliography{ref}

\end{document}